\documentclass{amsart}
\usepackage{amssymb,amsthm,amsmath,epsfig,latexsym,xypic,supertabular}
\usepackage{calc}
\usepackage{latexsym}
\usepackage{comment}
\usepackage{amscd,amssymb,subfigure,hyperref}
\usepackage[arrow,matrix,graph,frame,poly,arc,tips]{xy}
\usepackage{color}

\begin{document}

\newcommand{\mmbox}[1]{\mbox{${#1}$}}
\newcommand{\affine}[1]{\mmbox{{\mathbb A}^{#1}}}
\newcommand{\Ann}[1]{\mmbox{{\rm Ann}({#1})}}
\newcommand{\caps}[3]{\mmbox{{#1}_{#2} \cap \ldots \cap {#1}_{#3}}}
\newcommand{\N}{{\mathbb N}}
\newcommand{\V}{{\mathbb V}}
\newcommand{\Z}{{\mathbb Z}}
\newcommand{\Q}{{\mathbb Q}}
\newcommand{\R}{{\mathbb R}}
\newcommand{\KK}{{\mathbb C}}
\newcommand{\A}{{\mathcal A}}
\newcommand{\B}{{\mathcal B}}
\newcommand{\OO}{{\mathcal O}}
\newcommand{\C}{{\mathbb C}}
\newcommand{\pp}{{\mathbb P}}
\newcommand{\OS}{{T^d(X,p)}}

\newcommand{\val}{\mathop{\rm degree}\nolimits}
\newcommand{\charr}{\mathop{\rm char}\nolimits}
\newcommand{\ann}{\mathop{\rm ann}\nolimits}
\newcommand{\rank}{\mathop{\rm rank}\nolimits}
\newcommand{\Tor}{\mathop{\rm Tor}\nolimits}
\newcommand{\Ext}{\mathop{\rm Ext}\nolimits}
\newcommand{\Hom}{\mathop{\rm Hom}\nolimits}
\newcommand{\Sym}{\mathop{\rm Sym}\nolimits}
\newcommand{\im}{\mathop{\rm im}\nolimits}
\newcommand{\rk}{\mathop{\rm rk}\nolimits}
\newcommand{\codim}{\mathop{\rm codim}\nolimits}
\newcommand{\supp}{\mathop{\rm supp}\nolimits}
\newcommand{\coker}{\mathop{\rm coker}\nolimits}
\newcommand{\st}{\mathop{\rm st}\nolimits}
\newcommand{\lk}{\mathop{\rm lk}\nolimits}
\sloppy

\newtheorem{thm}{Theorem}[section]
\newtheorem*{thm*}{Theorem}
\newtheorem{defn}[thm]{Definition}
\newtheorem{prop}[thm]{Proposition}
\newtheorem{pref}[thm]{}
\newtheorem*{prop*}{Proposition}
\newtheorem{conj}[thm]{Conjecture}
\newtheorem{lem}[thm]{Lemma}
\newtheorem{rmk}[thm]{Remark}
\newtheorem{cor}[thm]{Corollary}
\newtheorem{notation}[thm]{Notation}
\newtheorem{exm}[thm]{Example}
\newtheorem{comp}[thm]{Computation}

\newcommand{\msp}{\renewcommand{\arraystretch}{.5}}
\newcommand{\rsp}{\renewcommand{\arraystretch}{1}}

\newenvironment{lmatrix}{\renewcommand{\arraystretch}{.5}\small
  \begin{pmatrix}} {\end{pmatrix}\renewcommand{\arraystretch}{1}}
\newenvironment{llmatrix}{\renewcommand{\arraystretch}{.5}\scriptsize
  \begin{pmatrix}} {\end{pmatrix}\renewcommand{\arraystretch}{1}}
\newenvironment{larray}{\renewcommand{\arraystretch}{.5}\begin{array}}
  {\end{array}\renewcommand{\arraystretch}{1}}

  \newenvironment{changemargin}[2]{%
\begin{list}{}{%
\setlength{\topsep}{0pt}%
\setlength{\leftmargin}{#1}%
\setlength{\rightmargin}{#2}%
\setlength{\listparindent}{\parindent}%
\setlength{\itemindent}{\parindent}%
\setlength{\parsep}{\parskip}%
}%
\item[]}{\end{list}}
\vskip -2in
\title[Algebraic aspects of homogeneous Kuramoto Oscillators]
{Algebraic aspects of homogeneous Kuramoto Oscillators}

\author[Heather A. Harrington]{Heather A. Harrington}
\thanks{Harrington supported by EPSRC EP/R018472/1, EP/R005125/1, EP/T001968/1, RGF 201074, UF150238, \\ \mbox{     }\mbox{     }\mbox{     }\mbox{     }and a Royal Society University Research Fellowship.}
\address{Mathematical Institute, University of Oxford, Oxford OX2 6GG, UK and \newline
\mbox{         }\mbox{       }\mbox{       }\mbox{  }Wellcome Centre for Human Genetics, University of Oxford, Oxford OX3 7BN, UK and \newline
\mbox{         }\mbox{       }\mbox{       }\mbox{  }Max Planck Institute of Molecular Cell Biology and Genetics, 01307 Dresden, Germany and \newline
\mbox{         }\mbox{       }\mbox{       }\mbox{  }Centre for Systems Biology Dresden, 01307 Dresden, Germany and \newline
\mbox{         }\mbox{       }\mbox{       }\mbox{  }Faculty of Mathematics, Technische Universitat Dresden, 01062 Dresden, Germany}
\email{\href{mailto:harrington@maths.ox.ac.uk}{harrington@maths.ox.ac.uk}}
\urladdr{\href{https://www.maths.ox.ac.uk/people/heather.harrington}{https://www.maths.ox.ac.uk/people/heather.harrington}}

\author[Hal Schenck]{Hal Schenck}
\thanks{Schenck supported by NSF DMS 2006410 and a Leverhulme Visiting Professorship.}
\address{Department of Mathematics,
Auburn University, Auburn, AL 36849 and \newline
\mbox{         }\mbox{       }\mbox{       }\mbox{  }Mathematical Institute, University of Oxford, Oxford UK}
\email{\href{mailto:hks0015@auburn.edu}{hks0015@auburn.edu}}
\urladdr{\href{http://webhome.auburn.edu/~hks0015/}%
{http://webhome.auburn.edu/~hks0015/}}

\author[Mike Stillman]{Mike Stillman}
\thanks{Stillman supported by NSF DMS 2001367 and a Simons Fellowship.}
\address{Department of Mathematics,
Cornell University, Ithaca, NY 14850 and \newline
\mbox{         }\mbox{       }\mbox{       }\mbox{  }Mathematical Institute, University of Oxford, Oxford UK}
\email{\href{mailto:mike@math.cornell.edu}{mike@math.cornell.edu}}
\urladdr{\href{https://math.cornell.edu/michael-e-stillman}{https://math.cornell.edu/michael-e-stillman}}


\begin{abstract}
\noindent We investigate algebraic characteristics of networks of coupled oscillators. Translating dynamics into a system of algebraic equations enables us to identify classes of network topologies that exhibit unexpected behaviors. Many previous studies focus on synchronization of networks having high connectivity, or of a specific type (e.g. circulant networks). We introduce the Kuramoto ideal; an algebraic analysis of this ideal allows us to identify features beyond synchronization, such as positive dimensional components in the set of potential solutions (e.g. curves instead of points). We prove sufficient conditions on the network structure for such solutions to exist. The points lying on a positive dimensional component of the solution set can never correspond to a linearly stable state. We apply this framework to give a complete analysis of linear stability for all networks on at most eight vertices. Furthermore, we describe a construction of networks on an arbitrary number of vertices having linearly stable states that are not twisted stable states. 
\vskip -.2in
\end{abstract}
\vskip -.2in
\maketitle
\vskip -.2in
\renewcommand{\thethm}{\thesection.\arabic{thm}}
\setcounter{thm}{0}
\vskip -.2in

\section{Introduction}

Dynamics on networks is an active area of mathematical research, with wide applicability in various fields including physics, engineering, biology and neuroscience. The study of dynamics on networks often involves understanding how the structure of the network influences the dynamics of the system. Dating back to 17th century, Dutch inventor and scientist, Christiaan Huygens, observed that two pendulum clocks hanging from a wall would synchronise their swing, which led to the study of coupled oscillators. 
Coupled oscillators appear in numerous applications, including biological and chemical networks \cite{Buck, matheny2019exotic, werner2005firefly,winfree1967biological}, power grids \cite{dokania2011low,dorfler1,dorfler2}, neuroscience \cite{Cumin}, spin glasses \cite{Kloumann}, and wireless communications \cite{dokania2011low,simeone2008distributed} (to name just a few). 

 A much studied question involving systems of oscillators is that of {\em synchronization}: under what conditions do the oscillators operate in harmony. One particularly well known instance of this involves flashing of fireflies \cite{Buck}. Our study focuses on understanding the algebra of Kuramoto oscillators, and using algebraic methods to analyze aspects of networks beyond synchronization. We give a complete description of networks of coupled oscillators with at most eight vertices that have linearly stable solutions. Our analysis also leads us to a theorem characterizing sufficient conditions for a network to have positive dimensional components in the set of potential solutions.
\subsection{Background on Kuramoto oscillators}
One of the most investigated oscillator models is due to Kuramoto \cite{kuramoto74,kuramoto85}. Let $G$ be a graph with $V$ vertices, representing the coupling of a system of oscillators, and for vertex $v_i$, let $G_i$ denote the set of vertices adjacent to $v_i$. The Kuramoto model is the system of $V$ equations below, where the phase $\theta_i$ is the angle at vertex $i$ at time $t$, $\omega_i$ is the natural frequency of the $i^{th}$ oscillator, and $K$ is the coupling strength:
\begin{equation}\label{Keqns}
\dot{\theta_i} = \omega_i + K \sum\limits_{v_j \in G_i} \sin(\theta_j-\theta_i).
\end{equation}
Perhaps the most frequently studied case is the {\em homogeneous} model, where the system consists of identically coupled phase oscillators; this allows us to assume $\omega_i=0$ and $K=1$, leading to equations
\begin{equation}\label{homogK}
\dot{\theta_i} = \sum\limits_{v_j \in G_i} \sin(\theta_j-\theta_i).
\end{equation}
We say that 
\[
{\bf \theta}=\{\theta_1^*,\ldots, \theta_V^*\} \in \R^V
\]
is an {\em equilibrium} if for all vertices $v_i$,
\begin{equation}\label{homogKsols}
0 = \sum\limits_{v_j \in G_i} \sin(\theta_j^*-\theta_i^*).
\end{equation}
As noted in \cite{Abdalla}, the homogeneous Kuramoto model is an instance of a gradient flow of an analytic potential function on a compact analytic manifold. Solutions of such flows are always convergent to a single point as a consequence of the \L{}ojasiewicz gradient inequality, see Lageman \cite{lageman}.

Therefore, for the homogeneous Kuramoto model, the natural question boils down from dynamics to algebraic geometry: what is the structure of the set of equilibrium states for a network of identical Kuramoto oscillators? The equilibrium states can be expressed as solutions of a system of {\em algebraic} equations, hence we also refer to equilibrium points as {\em solutions}. Solutions such that 
\[
\theta_i^*-\theta_j^* \in \{0, \pi\}
\]
are called {\em standard}. 

Our perspective yields new insights into the homogeneous Kuramoto model: for example, a continuous family of equilibrium states will correspond to a geometric object, and we will be interested in determining characteristics (for example, the dimension) of that object. 

Due to the rotational symmetry in the Kuramoto model, we always have at least a circle of equilibrium points. Changing coordinates so that $\theta_0 = 0$ allows us to ignore the trivial rotational symmetry. This yields a dichotomy, with some genuine isolated equilibrium points, and others part of a continuous family, and we refer to points in such a family as {\em positive dimensional solutions}. Of particular importance in understanding the long-term behavior of the system are those solutions which are {\em stable}.

\begin{defn}
Let $p$ be a solution to a system of first order ODEs. Then $p$ is (locally asymptotically) {\em stable} if there exists a small open neighborhood $N_\epsilon (p)$ of $p$ such that for any $q \in N_{\epsilon}(p)$, as the system moves forward in time from the point $q$, the solution converges back to the point $p$. 
\end{defn}

A system is said to {\em synchronize} if the only stable solution occurs when the $\theta_i^*$ are all equal. What properties of $G$ guarantee that a system synchronizes?

The {\em connectivity} of $G$  is defined as 
\[
\mu_c(G)= \min_{v \in G} \{ \val(v)/(V-1) \},
\]
and in \cite{taylor}, Taylor shows that if $\mu_c(G) \ge .94$ then the corresponding network synchronizes. Recent work of Ling-Xu-Bandeira \cite{ling} improves this bound to show synchronization when $\mu_c(G) \ge .79$. Non-standard stable solutions exist for certain configurations, such as when $G$ is a cycle of length $\ge 5$, which are the simplest avatars of the {\em circulant matrices} analyzed in \cite{wiley}. 

\medskip

\noindent For a system of ODEs given by $\dot{\theta_i}=f_i$, we study the related condition of {\em linear stability}: that for the Jacobian matrix   
\begin{equation}\label{Jacobian}
J(p)= \left[ \!
\begin{array}{ccc}
\frac{\partial(f_1)}{\partial \theta_1}(p) & \cdots &\frac{\partial(f_1)}{\partial \theta_n }(p) \\
\vdots & \ddots & \vdots\\
\frac{\partial(f_n)}{\partial \theta_1}(p) & \cdots& \frac{\partial(f_n)}{\partial \theta_n}(p) 
\end{array}\! \right]
\end{equation}
all eigenvalues have negative real part; as the homogeneous Kuramoto model has symmetric Jacobian, this means all eigenvalues are negative. For a homogeneous Kuramoto system and solution $p$, it is easy to see that the rows of $J(p)$ sum to zero, hence $\dim(\ker(J(p)) \ge 1$ and there is always one zero eigenvalue. In keeping with convention, we call a solution to Equation~\ref{homogK} {\em linearly stable} if it has all but one eigenvalue negative, because one of the equations can be eliminated.

In this paper, our goal is to understand the algebra of the solution sets to the Kuramoto equations appearing in Definitions~\ref{IGdef} and \ref{KuramotoAlg} below, and the connection to the topology of the graph $G$. In \cite{Mehta15}, the authors apply numerical algebraic geometry to the Kuramoto model and remark that the investigation of positive dimensional components is beyond the scope of their paper. Using a combination of numerical and symbolic methods we address this in \S 2. Algebraic tools are also utilized in \cite{Chen} and \cite{Coss}, but in a different context than this paper. Our focus is on the following two questions:
\begin{itemize}
    \item What graphs admit a positive dimensional set of possible solutions? 
    \item What graphs admit {\em exotic solutions}--linearly stable solutions where the $\theta_i$ are not all equal? One common type of exotic solution is a {\em twisted stable state}, where there is a periodic shift in the angles. A cycle $C_n$ with $n\ge 5$ always has twisted stable solutions, but these can also arise for noncycles, as in Example~\ref{Special8}. 
    \end{itemize}

\subsection{Conventions} All graphs we work with are {\em SCT} graphs: graphs that are simple (no loops or multiple edges), connected, and two-connected (all vertices of degree $\ge 2$). For a graph that does have a vertex $v_0$ of degree one, if $v_0v_1$ denotes the edge connecting $v_0$ to $G$, then choosing coordinates so $\theta_1 =0$ shows the only possible equilibria values for $\theta_0$ are $\{0, \pi\}$. Negative semidefinite matrices form a convex cone generated by rank one matrices, and a simple calculation shows that if $G’ = G \setminus v_0$ has an exceptional solution, so does $G$. We are investigating the other direction: when does adding a ``peninsular’’ vertex to a graph introduce new exotic solutions?
\subsection{Recent work} Some of the more frequently studied mathematical aspects of Kuramoto oscillators include (but are not limited to) the following:
\begin{itemize}
\item Synchronization, stability, connectivity: \cite{kassabov2, ling, lu,pikovsky2003synchronization,taylor}.
\item Special classes of graphs: random, dense, $k$-connected: \cite{canale2,deville2016phase,kassabov, townsend, zhang}.
\item Graphs with exotic stable states: \cite{abrams2,canale,wiley}. 
\end{itemize}
 \subsection{Results of this paper} The main advantage of an algebro-geometric approach is that it allows us to identify \emph{all} solutions to Equation~\ref{homogKsols}, in particular all linearly stable solutions. 
\begin{itemize}
\item In \S 2, we prove algebraic and algorithmic criteria for an SCT graph to have positive dimensional solutions to Equation~\ref{homogKsols}. The significance of this is that any solution on a positive dimensional component cannot be linearly stable.
We also prove that all standard solutions must lie on a specific irreducible algebraic variety,
the Segre embedding of $\pp^1 \times \pp^{n-1}$. This allows us to eliminate the $2^{V-1}-1$ unstable standard solutions from consideration, simplifying the analysis of potential non-standard solutions.

\vskip .05in
\item In \S 3, we use numerical algebraic geometry to identify SCT graphs with $V \in \{4,5,6,7,8\}$ vertices which admit an exotic solution. There are, respectively, $\{3,11,61,507,7442\}$ isomorphism classes of SCT graphs with $V \in \{4,5,6,7,8\}$. Of the graphs on eight vertices, we find 81 having exotic solutions, and every one of these--with one exception--has an induced cycle of length at least five. In general, gluing a cycle on five or more vertices (which has an exotic solution) to an arbitrary graph $G$ along a common edge does not preserve the exotic solution. We show an exotic solution exists for the graph $G'$ obtained by gluing all vertices of a graph $G$ to the two vertices of an edge of a five-cycle. 
\end{itemize}
In \S 4 we give examples of our computations illustrating several interesting phenomena, and in \S 5 we close with a number of questions for further research that are raised by our results.
\subsection{Encoding the graph $G$ algebraically}
For a system of Kuramoto oscillators on a graph $G$ with $V=n$, label the vertices with $\{ 0, \ldots, n-1 \}$. We translate from the trigonometric relations of Equation~\ref{homogKsols} to algebraic relations via the substitution $x_i=\sin(\theta_i)$ and $y_i=\cos(\theta_i)$. This translation yields constraints expressed as {\em an ideal}:
\[
I_\theta = \langle x_0^2+y_0^2-1, \ldots, x_{n-1}^2+y_{n-1}^2-1 \rangle.
\]
The solutions to these equations capture the relations 
$\sin^2(\theta_i)+\cos^2(\theta_i)=1$. We also need to encode the dynamics of the graph, described by a polynomial equation at each vertex of $G$:
\begin{defn}\label{IGdef}
For $v_i \in G$, let $G_i$ denote the set of vertices $\{v_{i_1},\ldots, v_{i_k} \}$  adjacent to $v_i$. Then since $\sin(\theta_j-\theta_i)=\sin(\theta_j)\cos(\theta_i)-\sin(\theta_i)\cos(\theta_j)=x_jy_i-x_iy_j$, at vertex $v_i$, we have the equation
\[
f_i=\sum\limits_{v_j \in G_i} x_jy_i-x_iy_j
\]
For a graph $G$ on vertex set $\{0,\ldots, n-1\}$, define $I_G = \langle f_0, \ldots, f_{n-1} \rangle$, with the $f_i$ as above.
\end{defn}

\begin{exm}\label{ex1}
For the graph $G$ on five vertices depicted below, the ideal $I_G$ is given by
\[ 
\begin{array}{ccc}
I_G &= &  \langle x_{2}y_{0}+x_{3}y_{0}+x_{4}y_{0}-x_{0}y_{2}-x_{0}y_{3}-x_{0}y_{4},\\
  & &-x_{2}y_{0}-x_{2}y_{1}+x_{0}y_{2}+x_{1}y_{2},\\
  & &-x_{3}y_{0}-x_{3}y_{1}+x_{0}y_{3}+x_{1}y_{3},\\
  & &-x_{4}y_{0}-x_{4}y_{1}+x_{0}y_{4}+x_{1}y_{4},\\
  & &x_{2}y_{1}+x_{3}y_{1}+x_{4}y_{1}-x_{1}y_{2}-x_{1}y_{3}-x_{1}y_{4}\rangle
         \end{array}
\]
\begin{figure}[h]
\vskip -.2in
\includegraphics[width=6in]{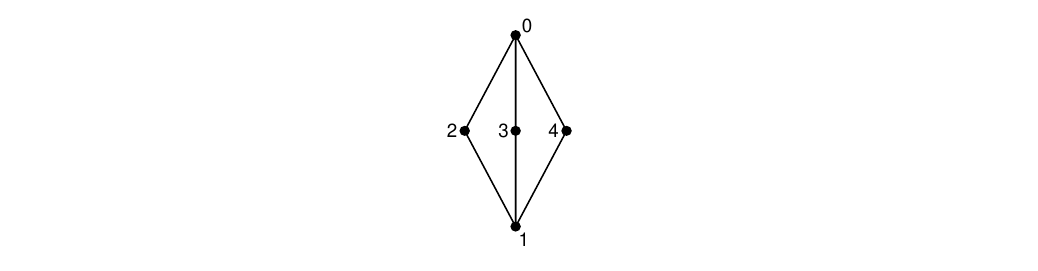} 
\vskip -.2in
\caption{The graph $G$ and ideal $I_G$}
\end{figure}

\end{exm}

\begin{defn}\label{KuramotoAlg}
The {\em Kuramoto variety} is
the set of common solutions in $\mathbb{C}^{2n}$ to the
equations 
\[ f_0 = f_1 = \ldots f_{n-1} = x_0^2 + y_0^2 - 1 = \ldots = x_{n-1}^2 + y_{n-1}^2 - 1 = 0.\]
The set of polynomials above defines an algebraic object,
the {\em Kuramoto oscillator ideal}
\[
I_K = I_\theta + I_G,
\]
which is an algebraic encoding of Equation~\ref{homogKsols}.
The common zeros of all polynomials in the ideal $I_K$
is the Kuramoto variety which we denote by
$\V(I_K)$. 
\end{defn}

\noindent The next section is devoted to the study of algebraic properties of the ideal $I_G$.
\section{Algebra of Kuramoto oscillators: determinantal equations}
\noindent As in \S1, for a graph $G$ and vertex $v_i \in G$, let $G_i$ denote the set of vertices adjacent to $v_i$, and suppose $G_i= \{v_{i_1}, \ldots, v_{i_k}\}$. To simplify notation, let 
\[
\begin{array}{c}
x_{i\bullet} \mbox{ denote the sum }x_{i_1}+ \cdots +x_{i_k}\\
y_{i\bullet} \mbox{ denote the sum }y_{i_1}+ \cdots +y_{i_k}
\end{array}
\]
\begin{lem}\label{IGlem}
With notation as above, 
\[
I_G = \Big \langle \det \left[ \!
\begin{array}{cc}
x_0 & x_{0\bullet}\\
y_0 & y_{0\bullet}
\end{array}\! \right], \ldots,
\det \left[ \!
\begin{array}{cc}
x_{n-1} & x_{n-1\bullet}\\
y_{n-1} & y_{n-1\bullet}
\end{array}\! \right] \Big \rangle,
\]
and $I_G$ has $V-1$ minimal generators $($rather than the expected number $V)$.
\end{lem}
\begin{proof}
Since 
\[
\det \left[ \!
\begin{array}{cc}
x_i & x_{i\bullet}\\
y_i & y_{i\bullet}
\end{array}\! \right] = \sum\limits_{v_j \in G_i} x_iy_j-x_jy_i
\]
the first result follows from Definition~\ref{IGdef}. To see that $I_G$ has $V-1$ minimal generators, first observe that 
$f_0+f_1+\cdots +f_{n-1}=0$, so one of the $f_i$ is redundant. To streamline notation let $n=V$, so $I_G$ has  generators (including the one non-minimal generator) $\{f_1,\ldots, f_n\}$, which we write as a matrix product
\begin{equation}\label{JacLemma}
\left[ \!
\begin{array}{ccc}
\frac{\partial(f_1)}{\partial x_1} & \cdots &\frac{\partial(f_1)}{\partial x_n } \\
\vdots & \ddots & \vdots\\
\frac{\partial(f_{n})}{\partial x_1} & \cdots& \frac{\partial(f_{n})}{\partial x_n} 
\end{array}\! \right] \cdot
\left[ \!
\begin{array}{c}
x_1\\
\vdots \\
x_n 
\end{array}\! \right]
=
\left[ \!
\begin{array}{c}
f_1\\
\vdots \\
f_{n} 
\end{array}\! \right]
\end{equation}
\noindent Let $[Y]$ denote the matrix on the left of Equation~\ref{JacLemma}; the entries of $[Y]$ are: 
\begin{equation}\label{Ymat}
[Y]_{ij}= \begin{cases}
  -\sum\limits_{v_j \in G_i}y_j & \mbox{ if } i=j \\
   y_j & \mbox{ if } i \ne j \mbox{ and } v_j \in G_i\\
   0   & \mbox{otherwise}
\end{cases}
\end{equation}
Let $[Y]_{{\bf y}=1}$ be the matrix obtained by setting $\{y_1 =y_2=\cdots =y_n=1\}$. Then  
\[
[Y]_{{\bf y}=1} = -L_G,
\]
where $L_G$ is the graph Laplacian, which has rank $n-1$ (e.g. Lemma 3.4.5 of \cite{TDAbook}). As $f_0+f_1+\cdots +f_{n-1}=0$, the rank of $[Y]$ is at most $n-1$, and since rank drop is a Zariski closed condition, the argument above shows the rank of $[Y]$ is exactly $n-1$. Hence there is exactly one dependency on the $n$ generators of $I_G$, as we wanted to show. \end{proof}
\subsection{The Segre variety}
In classical algebraic geometry, the Segre variety (\cite{E}, Exercise 13.14), is the image of the map
\[
\mathbb{S}_{s,t}: \pp^s \times \pp^t \longrightarrow \pp^{st+s+t},
\]
defined by 
\[
[a_0:\cdots:a_s] \times [b_0:\cdots:b_t] \mapsto [a_0b_0: a_0b_1:\cdots: a_0b_t:a_1b_0: \cdots :a_sb_t].
\]
The Segre variety with $s=1$ and $t=V-1$ (henceforth written simply as $\Sigma$) plays an important role in the study of Kuramoto oscillators: in \S 2.4 we show that all standard solutions lie on $\Sigma$. 

\noindent For $\pp^1 \times \pp^{V-1}$, the target space of the Segre map $\mathbb{S}_{1,V-1}$ defined above is $\pp^{2V-1}$, and the ideal $I_\Sigma$ of polynomials vanishing on the image $\Sigma$ is
\begin{equation}\label{SegreDef}
I_\Sigma = I_2\left[ \!
\begin{array}{cccc}
x_0 & x_1& \cdots x_{V-1}\\
y_0 & y_1& \cdots y_{V-1}
\end{array}\! \right], \mbox{ where } I_2 \mbox{ denotes the }2 \times 2 \mbox{ minors of the matrix}.
\end{equation}

\noindent Since $I_\Sigma$ is generated by the ${ V \choose 2}$ polynomials 
\[
\det \left[ \!
\begin{array}{cc}
x_{i} & x_{j}\\
y_{i} & y_{j}
\end{array}\! \right]\!,
\]
this means that the $V-1$ generators of $I_G$ are sums of the generators of $I_\Sigma$. 
To analyze $I_G$ more deeply, we need some algebraic geometry.
\subsection{Algebraic geometry interlude}
\noindent In this section, we describe some key algebro-geometric quantities associated to the ideal $I_G$; see \cite{E} or \cite{S} for additional details.

\begin{defn}\label{codimDef} For $I$ an ideal in a polynomial ring $R = \KK[x_1, \ldots, x_m]$, the 
 {\em complementary dimension} (or codimension) of $I$ is 
\[
\codim(I) = m-\dim \V(I),
\]
where $\V(I) \subset \KK^m$ is the set of common zeros of the polynomials defining $I$.
\end{defn}

\begin{exm}\label{CI}
To gain geometric intuition for the definition above, notice that 
a minimal generator $f \in I$ has solution set $\V(f)$ which is a hypersurface. If $I$ is minimally generated by $\{f_1,\ldots, f_d\}$ and $\codim(I)=d$, then each hypersurface $\V(f_i)$ drops the dimension of the solution space $\V(I)$ by one, and $I$ is called a {\em complete intersection}. In general, $\codim(I)$ is $\le $ the number of generators of $I$, and the solution space $\V(I)$ is called the variety of the ideal $I$.
\end{exm}

\noindent After dehomogenizing (setting some variable equal to one) the system of equations $I_G$, in order for the Kuramoto variety to consist of a finite number of solutions, the codimension of $I_K$ must be $2V-1$. Since there are exactly $V$ equations in $I_\theta$, a necessary condition for $\V(I_K)$ to be finite is that $\codim(I_G) = V-1$. By Lemma~\ref{IGlem}, $I_G$ has exactly $V-1$ generators, so we have proved
\begin{thm} $\V(I_K)$ has a finite set of solutions only if $I_G$ is a complete intersection.
\end{thm}

 \noindent But in general, $I_G$ is not a complete intersection, as we will see in Example~\ref{3doubles}.
 
\begin{defn}\label{MinimalPrime}
An ideal $P \subsetneq R$ is called prime if $fg \in P$ implies $f \in P$ or $g \in P$ (or both). For an ideal $I \subset R$, a prime ideal $P$ containing $I$ is a {\em minimal prime} of $I$ if there is no other prime ideal $Q$ such that $I \subseteq Q \subsetneq P$.
\end{defn}

\noindent The set of common zeros $\V(P)$ of a prime ideal $P$ is {\em irreducible}:  $\V(P) \ne \V(I_1) \cup \V(I_2)$, for any non-empty $\V(I_i) \subsetneq \V(P)$.
For $I \subset R$ the variety $\V(I)$ has a unique finite minimal {\em irreducible decomposition}
\[ \V(I) = \bigcup\limits_{i=1}^d \V(P_i), \mbox{ where the }P_i \mbox{ are prime}.\]
The $P_i$ above are the {\em minimal primes} of $I$, and the $\V(P_i)$ the {\em irreducible components} of $\V(I)$.

\begin{exm}\label{3doubles}
For the graph $G$ on five vertices appearing in Example~\ref{ex1}, 
the ideal $I_G$ is the intersection $I_G = P_1 \cap P_2 \cap P_3$, where the
$P_i$'s are prime ideals, described below.
\[
\begin{array}{ccc}
P_1 & = &   
\begin{array}{c}
\langle y_{0}+y_{1},\\
        \;x_{0}+x_{1}, \\
        x_{2}y_{1}+x_{3}y_{1}+x_{4}y_{1}-x_{1}y_{2}-x_{1}y_{3}-x_{1}y_{4}\rangle 
  \end{array}\\
 & & \\
P_2 & = & 
\begin{array}{c} 
\langle y_{2}+y_{3}+y_{4},\\
        \;x_{2}+x_{3}+x_{4},\\
        x_{4}y_{3}-x_{3}y_{4},\,x_{4}y_{0}+x_{4}y_{1}-x_{0}y_{4}-x_{1}y_{4},\,x_{3}y_{0}+x_{3}y_{1}-x_{0}y_{3}-x_{1}y_{3}\rangle
 \end{array}\\
 & & \\
 P_3 & = & \mbox{the Segre ideal } I_\Sigma \mbox{ of Equation }\ref{SegreDef}. 
\end{array}
\]
Lemma~\ref{IGlem} is the key to analyzing the irreducible decomposition: write the generators of $I_G$ as 
\[
\begin{array}{ccc}
I_G &=& \Big \langle \det \left[ \!
\begin{array}{cc}
x_2 &x_0 +x_1\\
y_2 &y_0 +y_1
\end{array}\! \right], 
\det \left[ \!
\begin{array}{cc}
x_3 &x_0 +x_1\\
y_3 &y_0 +y_1
\end{array}\! \right], 
\det \left[ \!
\begin{array}{cc}
x_4 &x_0 +x_1\\
y_4 &y_0 +y_1
\end{array}\! \right], \cr\\
& & \!\!\!\!\! \!\!\!\!\! \!\!\!\!\! \!\!  \!\!\det \left[ \!
\begin{array}{cc}
x_0 & x_2+x_3+x_4\\
y_0 & y_2+y_3+y_4
\end{array}\! \right],
\det \left[ \!
\begin{array}{cc}
x_1 & x_2+x_3+x_4\\
y_1& y_2+y_3+y_4
\end{array}\! \right]
\Big \rangle.
\end{array}
\]
Component (1) has codimension $3 < V-1 = 4$, so $I_G$ is not a complete intersection. It is easy to see this from the determinantal description of the generators above: when 
\begin{equation}\label{2lins}
y_0+y_1=0 = x_0+x_1,
\end{equation} 
the first three determinantal equations all vanish. Summing the two remaining generators we can use Equation~\ref{2lins} to eliminate one of them; the resulting solutions are the zeros of Component (1). 
\end{exm}
\subsection{Low codimension components of $I_G$}
It turns out that Example~\ref{3doubles} provides insight into obtaining a more general description of irreducible components of $I_G$. In \cite{canale}, Canale-Monz{\'o}n define two vertices as {\em twins} if they have the same set of adjacent vertices. This suggests looking at triplets, quadruplets, and so on, so we define:
\begin{defn}\label{Klet}
A k-let is a set $S$ of $k$ distinct vertices of $G$ such that $v, w \in S \longrightarrow G_v = G_w$. 
\end{defn}
\begin{thm}\label{PDsolns}
If $G$ has a k-let with $k \ge 3$, then $\codim(I_G) \le V-k+1$.
\end{thm}
\begin{proof}
It suffices to prove that $I_G$ 
is contained in an ideal of codimension at most $V-k+1$.
First, after relabelling, we may suppose our $k-let$ is $\{v_0, \ldots, v_{k-1}\}$, hence $G_{0} = G_{1} = \cdots G_{k-1}$. So the determinantal equations for $I_G$ take the form
\[
\begin{array}{ccc}
I_G &=& \Big \langle \det \left[ \!
\begin{array}{cc}
x_0 &x_{0 \bullet}\\
y_0 &y_{0 \bullet}
\end{array}\! \right], 
\det \left[ \!
\begin{array}{cc}
x_1 &x_{1 \bullet}\\
y_1 &y_{1 \bullet}
\end{array}\! \right], \ldots,
\det \left[ \!
\begin{array}{cc}
x_{k-1} &x_{k-1 \bullet}\\
y_{k-1} &y_{k-1 \bullet}
\end{array}\! \right], f_k,\ldots, f_{V-2}
\Big \rangle.
\end{array}
\]
Note that we have used Lemma~\ref{IGlem} to reduce to $V-1$ generators.
Since the first $k$ vertices are a k-let, for $i \in \{0,\ldots, k-1\}$ the linear forms $x_{i \bullet}$ are equal, and similarly for $y_{i \bullet}$. Therefore the vanishing of the two linear forms $\{x_{0 \bullet}, y_{0 \bullet}\}$ causes the first $k$-equations defining $I_G$ to vanish. So 
\[ I_G \subset \langle x_{0 \bullet}, y_{0 \bullet}, f_{k}, \ldots, f_{V-2} \rangle,
\]
an ideal with $V-k+1$ generators,
and hence of codimension at most $V-k+1$.

\end{proof}
\noindent Example~\ref{3doubles} contains a 3-let, and $V=5$, so  $I_G$ is of codimension $\le 5-3+1 = 3$, hence $\dim(\V(I_G))\ge 7$.  Setting $x_0 = 1$ and $y_0 = 0$, and adding in the remaining four trigonometric equations $x_i^2 + y_i^2 - 1 = 0$,  yields (at least) a one dimensional set of solutions.

\begin{cor}\label{IKdim}
If $G$ has a k-let with $k \ge 3$, then $\V(I_K)$ has positive dimension. 
\end{cor}
\subsection{The Segre variety and standard solutions}
\noindent All standard solutions lie on the Segre variety $\Sigma$ with $s=1$ and $t=V-1$. This follows because by a change of variables we can assume that $\theta_0^*=0$. The standard solutions satisfy $\theta_i^*-\theta_j^* \in \{0,\pi\}$, so the $x_i$ are all zero. Since $\Sigma$ is defined by the ideal appearing in Equation~\ref{SegreDef}, the inclusion of the standard solutions in $\Sigma$ follows. The codimension of $\Sigma$ is $V-1$ (see Chapter 2 of \cite{CLS}), so adding the $V$ equations of $I_\theta$ to $I_\Sigma$ produces an ideal of codimension at most $2V-1$. Using the change of variables $\theta_0^*=0$ above eliminates the equation $x_0^2+y_0^2-1$ as well as the variables $x_0$ and $y_0$, yielding a system of $2V-2$ equations in $2V-2$ unknowns. For the next theorem, recall that a projective variety $X \subseteq {\mathbb P}^k$ also defines as an affine variety (known as the {\em affine cone}) in ${\mathbb C}^{k+1}$.
\pagebreak

\begin{thm}\label{SegreAssPrime} 
The affine cone over $\Sigma$ is an irreducible component of $\V(I_G) \subset \mathbb{C}^{2V}$.
\end{thm}
\begin{proof}

To prove the theorem, it suffices to show that
$I_\Sigma$ is a minimal prime of $I_G$.  Note that as $I_\Sigma$ is a prime ideal with $\codim(I_\Sigma) = n-1$, if $I_G$ is a complete intersection, Example~\ref{CI} shows that 
\[
\codim(I_G)=n-1.
\]
Since $I_G \subseteq I_\Sigma$, this means that $I_\Sigma$ is a minimal prime of $I_G$ when $\codim(I_G)=n-1$. 

Therefore to conclude the proof, we need to show that when $I_G$ has codimension $\le n-2$, $I_\Sigma$ is still minimal over $I_G$. If this were not the case, then there would exist a prime ideal $P$, with $\codim(P) < \codim(I_\Sigma)$ such that 
\[
I_G \subseteq P \subsetneq I_\Sigma \Longrightarrow \Sigma \subsetneq \V(P) \subseteq \V(I_G). 
\]
Our strategy is to consider the tangent spaces of the corresponding varieties. From the containments above, for any point $p \in \V(I_\Sigma)$, 
\[
T_p(\Sigma) \subsetneq T_p(\V(P)) \subseteq T_p(\V(I_G)). 
\]
Since $\codim(P) < \codim(I_\Sigma)$, $\dim \V(P) > \dim \Sigma$. Since $\Sigma$ is smooth, at any point $p$ of $\Sigma$, 
\[
\dim T_p(\Sigma) = \dim \Sigma = n+1 = \dim \ker(J_p(I_\Sigma)),
\]
where we compute the dimension as an affine variety, and $J_p(I)$ denotes the Jacobian of an ideal, evaluated at a point $p$. Hence for any point $p$, 
\[
\rank(J_p(I_\Sigma)) = n-1. 
\]
\vskip .1in
\noindent Now consider the point where $y_i=1=x_i$ for all $i$ which we denote {\bf 1}. Observe that 
\[
J(I_G) = [Y | X],
\]
where $[Y]$ is as in Equation~\ref{Ymat} and $[X]$ is defined by the same formula as $[Y]$, but with $x_i$ substituted for $y_i$. By our earlier computations, 
\[
T_{\bf 1}(\V(I_G)) = \ker(J_{{\bf 1}}(I_G)) = \ker([-L_G | L_G]).
\]
As $\rank(L_G) = n-1$, this means 
\[
T_{{\bf 1}}(\Sigma) = T_{{\bf 1}}(\V(I_G)),
\]
which contradicts the existence of $\V(P)$ properly containing $\Sigma$ and contained in $\V(I_G)$. 
\end{proof}
\begin{cor}
The points of $\Sigma \cap \V(I_\theta)$ lie in the Kuramoto variety $\V(I_K)$.
\end{cor}
\begin{proof}
By Theorem~\ref{SegreAssPrime}, $\Sigma \subseteq \V(I_G)$; since $\V(I_K) = \V(I_G) \cap \V(I_\theta)$ the result follows.    
\end{proof}
\subsection{Case study: the complete graph $K_n$}
It follows from the description of $I_G$ appearing in Lemma~\ref{IGlem} that when defining the ideal $I_G$, we may include $v$ in the set $G_v$. 
\begin{exm}\label{KnExample}
From our observation above, when $G=K_n$, adding $v$ to each set $G_v$ yields
%
\[
L_1 = \sum\limits_{i=0}^{n-1}x_i = x_{0\bullet} = \cdots =x_{n-1\bullet} \mbox{ and } L_2 = \sum\limits_{i=0}^{n-1}y_i = y_{0\bullet} = \cdots =y_{n-1\bullet}.
\]
Hence when $G=K_n$, we have
\begin{equation}\label{KnEquation}
\begin{array}{ccc}
I_{K_n} &= &\Big \langle \det \left[ \!
\begin{array}{cc}
x_0 & L_1\\
y_0 & L_2
\end{array}\! \right], \ldots,
\det \left[ \!
\begin{array}{cc}
x_{n-1} & L_1\\
y_{n-1} & L_2
\end{array}\! \right] \Big \rangle\\& =&\Big \langle \left[ \!
\begin{array}{cc}
L_2 & -L_1
\end{array}\! \right] \cdot \left[ \!
\begin{array}{ccc}
x_0 & \cdots & x_{n-1}\\
y_0 & \cdots & y_{n-1}
\end{array}\! \right] \Big \rangle
\end{array}
\end{equation}
By Theorem~\ref{PDsolns}, $I_L=\langle L_1, L_2 \rangle$ contains $I_{K_n}$, and by Theorem~\ref{SegreAssPrime} $I_\Sigma$ is a minimal prime, hence
\[
I_{K_n} \subseteq I_\Sigma \bigcap I_L. 
\]
\noindent In fact, $\V(I_{K_n}) = \Sigma \bigcup \V(I_L)$, because if $p \in \V(I_L)^c$, then the matrix multiplication in Equation~\ref{KnEquation} composes to zero exactly when $p \in \Sigma$, and otherwise $p \in \V(I_L)$.
The Jacobian matrix has $(1,1)$ entry $-x_1-x_2-\cdots -x_{n-1}$; if we set $x_0 = 1$ and $y_0 = 0$ $($so $\theta_0 = 0)$, then on $\V(I_L)$ since $L_1=0$ the $(1,1)$ entry is $x_0=1$. So $J(p)$ cannot be negative semidefinite for any $p \in \V(I_L)$. 

In Theorem 4.1 of \cite{taylor}, Taylor proves that $K_n$ synchronizes; the argument above shows that no points $p \in \V(I_L)$ exist such that $J(p)$ has one zero eigenvalue, and all other eigenvalues $<0$, so any linearly stable point is one of the standard solutions lying on $\Sigma$; Taylor proves that the only stable standard solution is when all angles are equal.
\end{exm}
\noindent Example~\ref{KnExample} shows that in the set of isomorphism classes of SCT graphs on $n$ vertices, there will always be at least one graph--$K_n$--having a solution set with positive dimension. How common is this phenomenon? Below is a summary of our computational results. The code and a tutorial on using it is part of the {\tt Macaulay2} software, available at {\tt http://https://macaulay2.com}.
\vskip .1in
\begin{table}[h]
\centering

\begin{tabular}{|c|c|c|c|c|}
\hline $\stackrel {\sharp}{\mbox{vertices}}$ & $\stackrel{\sharp \mbox{ iso. classes}}{\mbox{of SCT graphs}}$ & $\stackrel{\sharp \mbox{ SCT graphs with}}{\codim(I_G)\le n\!-\!2}$ & $\stackrel{\sharp \mbox{ SCT graphs}}{\mbox{with exotic solns.}}$ &$\stackrel{\sharp \mbox{ SCT graphs with }}{\stackrel{\mbox{exotic solns. and }}{\codim(I_G)\le n\!-\!2}}$\\
\hline $4$ & $3$       & $1$     & $0$& 0\\
\hline $5$ & $11$   & $4$  & $1$ &0\\
\hline $6$ & $61$  & $15$ & $2$ &0\\
\hline $7$ & $507$  & $91$ & $9$ &0\\
\hline $8$ & $7442$ & $809$ & $81$ &6\\
\hline
\end{tabular}
\vskip .1in
\caption{Numerical algebraic geometry results for SCT graphs  }
\vskip -.2in
\end{table}
\section{Graphs with exotic solutions}
\noindent Cycle graphs always admit exotic solutions, which is a reflection of a more general result. A {\em circulant graph} is a graph whose automorphism group acts transitively on the vertices. As a result, the adjacency matrix of the graph is a circulant matrix: each row is a cyclic shift by one position of the row above it. Wiley-Strogatz-Girvan show that circulant graphs admit exotic solutions \cite{wiley}. Such graphs have the highest connectivity known for graphs with exotic solutions, with $\mu(G) \sim .68$. In this section, we investigate exotic states for graphs with a small number of vertices using numerical algebraic geometry. All angles will be represented in radians.
\subsection{Exotic solutions on at most 7 vertices}
Numerical algebraic geometry indicates that graphs on five or six vertices which have an exotic solution are either cycles, or a pentagon and triangle glued on a common edge:
\begin{figure}[h]
\vskip -1.9in
\includegraphics[width=3in]{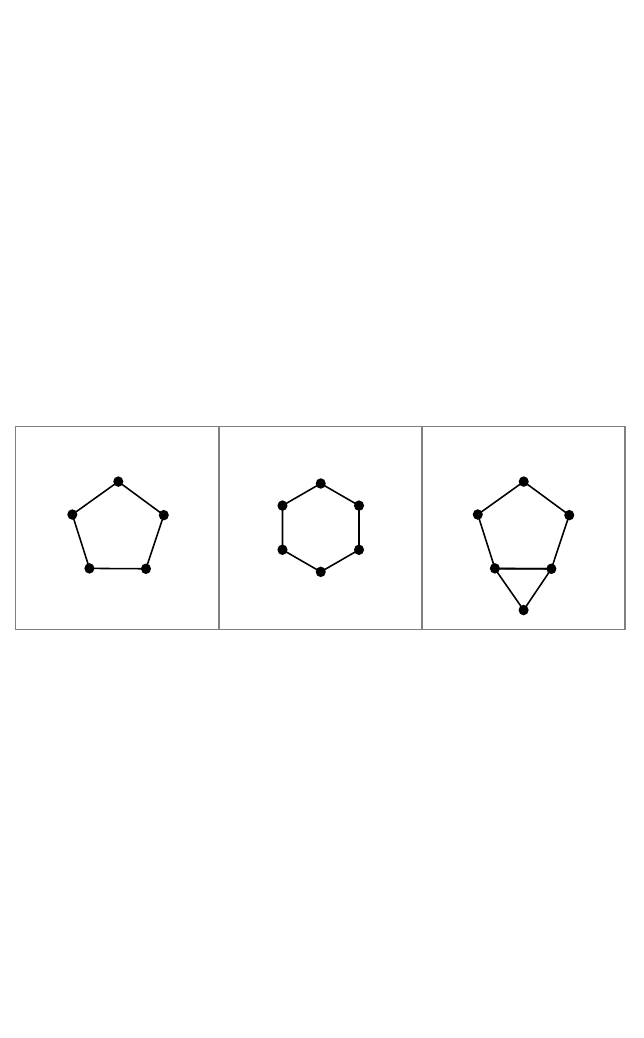}
\vskip -1.9in
\caption{Graphs with exotic solutions, $V=5$ or $6$}
\end{figure}

\noindent  While there are only 11 isomorphism classes of graphs on five vertices, on six vertices, there are 61 isomorphism classes, and for seven vertices, there are 507 isomorphism classes. And indeed, on seven vertices, things become more interesting: numerical algebraic geometry identifies 9 such graphs, shown in Figure~\ref{fig:3}.
\begin{figure}[h]
\vskip -.1in
\includegraphics[width=4.5in]{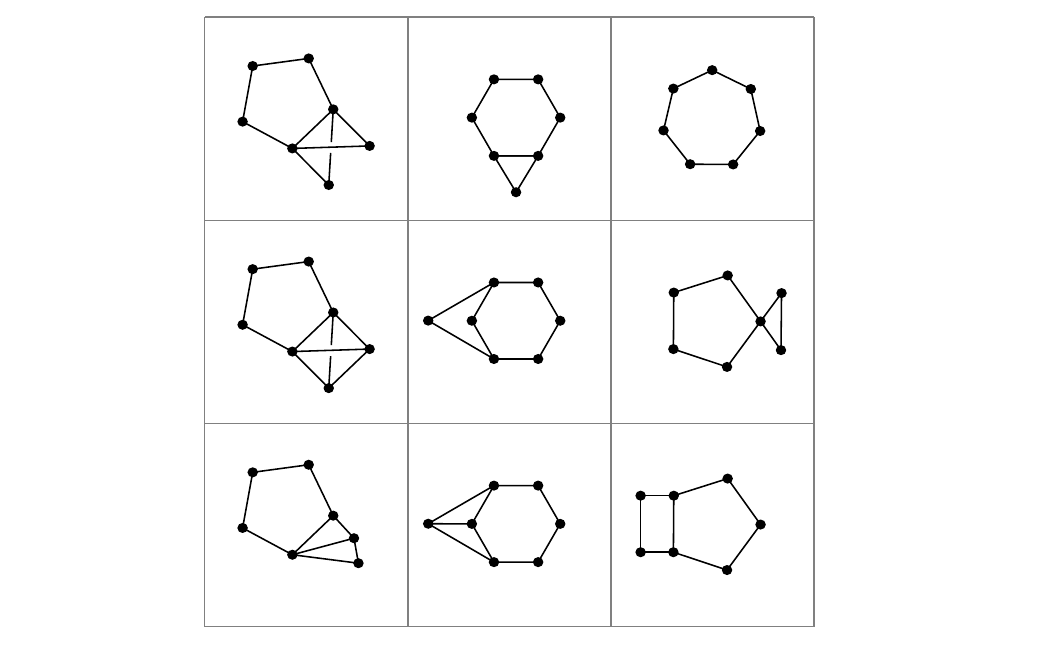}
\caption{Graphs with exotic solutions, $V=7$}
\label{fig:3}
\end{figure}

\noindent What is noteworthy about this is that every one of these graphs has a chordless $n$-cycle $C_n$, with $n \ge 5$. And in fact, this is also the case for {\em all but one} of the graphs on eight vertices having exotic solutions.

\subsection{Exotic solutions, 8 vertices}
For eight vertices, there are 7442 isomorphism classes of SCT (simple, connected, all vertices of degree $\ge 2$) graphs, and numerical algebraic geometry identifies 81 which have exotic solutions.
\begin{exm}\label{Special8}
All of the graphs on 8 vertices having exotic solutions have a chordless $n \ge 5$ cycle, with one exception (see Figure~\ref{fig:4}).
\begin{figure}[h]
\vskip -.2in
\includegraphics[width=6.0in]{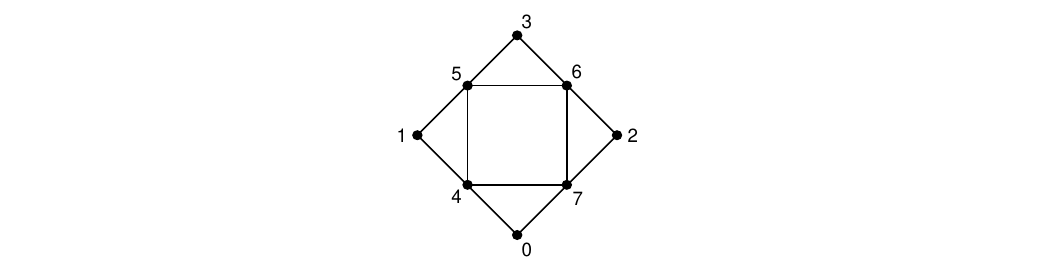}
\caption{The exceptional 8-vertex graph with exotic solutions}
\label{fig:4}
\end{figure}

\noindent The graph in Figure~\ref{fig:4} admits two exotic solutions, which are both twisted stable states. Recall that $x_i = \cos(\theta_i)$ and $y_i = \sin(\theta_i)$; rounded to 2 decimals, the solutions are below. 

\noindent Note that the top row is the standard stable synchronized solution.
\[
\left[ \!\begin{array}{cccccccccccccccc}
x_0&x_1&x_2&x_3&x_4&x_5&x_6&x_7&y_0&y_1&y_2&y_3&y_4&y_5&y_6&y_7\\
1&1&1&1&1&1&1&1&0&0&0&0&0&0&0&0\\
1&0&0&-1&.71&-.71&-.71&.71&0&-1&1 & 0&-.71&-.71&.71&.71\\
1&0&0&-1&.71&-.71&-.71&.71&0&1 &-1&0&.71 &.71 &-.71&-.71
\end{array}\! \right]
\]
The corresponding angles $\theta$ (note labelling above) are
\[
\left[ \!\begin{array}{cccccccc}
0&0&0&0&0&0&0&0\\
0&\frac{3\pi}{2}&\frac{\pi}{2}&\pi&\frac{7\pi}{4}&\frac{5\pi}{4}&\frac{3\pi}{4}&\frac{\pi}{4}\\
0&\frac{\pi}{2}&\frac{3\pi}{2}&\pi&\frac{\pi}{4}&\frac{3\pi}{4}&\frac{5\pi}{4}&\frac{7\pi}{4}
\end{array}\! \right]
\]
We discuss these computations more in Example~\ref{comp1}. For the ideal $I_G$,
\[
I_G = J \bigcap I_\Sigma, 
\]
where $J$ is a prime ideal generated by seven quadrics, which are exactly the quadrics defining $I_G$, and seven sextics. These sextics are all quite complicated;  even after reducing them modulo the quadrics in $I_G$, the smallest has 170 terms and the largest 373 terms. 

To prove this is the irreducible decomposition, we first compute the ideal quotient $I_G : I_\Sigma$, which yields an ideal $J$.  Using {\it Macaulay2} we verify that $J$ is prime, and that $J \cap I_\Sigma = I_G$.

\end{exm} 
\subsection{Constructing Graphs with Exotic Solutions}
We begin with an observation of  Ling-Xu-Bandeira in \S 2.1 of \cite{ling}. 
\begin{lem}\label{glue2}
If we allow edges to have negative weights, then for the homogeneous Kuramoto system of a graph $G$,  the Jacobian matrix appearing in Equation~\ref{Jacobian} is the weighted graph Laplacian with weight matrix 
\[
a_{ij}(\cos(\theta_i-\theta_j))_{1 \le i \le j \le n},
\]
\vskip .1in
\noindent where the $a_{ij}$ are coefficients of the adjacency matrix of $G$. In particular, if $p=(\theta_0,\ldots,\theta_{n-1})$ is an equilibrium point such that $\cos(\theta_i-\theta_j)> 0 $ for all edges, then the point $p$ is linearly stable. 
\end{lem}
\begin{proof}
The standard proof that the graph Laplacian of a connected graph has one zero eigenvalue and the remaining eigenvalues are greater than zero uses the fact that the Laplacian factors as
\[
L_G = B_G \cdot B_G^T,
\]
\vskip .1in
\noindent where $B_G$ is an oriented edge-vertex adjacency matrix (the orientation chosen is immaterial). Therefore choosing $\widetilde{B_G}$ to be weighted with weight of the edge $e_{ij} = \sqrt{\cos(\theta_i -\theta_j)}$ provides the necessary adjustment to take the weighting into account. As a consequence, as long as the weightings $\cos(\theta_i-\theta_j)$ are all positive, the resulting Laplacian satisfies the condition for linear stability: if 
\[
\theta_i - \theta_j \in \Big(-\frac{\pi}{2}, \frac{\pi}{2}\Big). 
\]
then the system is linearly stable at $p$.
\end{proof}
\begin{rmk} Computations using numerical algebraic geometry identified a pair of graphs with exotic solutions (hence, solutions which are linearly stable), but where not all the angles satisfy $\theta_i - \theta_j \in (-\frac{\pi}{2}, \frac{\pi}{2})$. We analyze these graphs in Example~\ref{comp4}.
\end{rmk}

\noindent In constructing and analyzing examples of exotic solutions, the following lemma will be useful.
\begin{lem}\label{2valent} Let $v$ be a vertex of degree two, with $b$ denoting the angle at $v$, and $a$, $c$ the angles at the two vertices adjacent to $v$. A solution to the homogeneous Kuramoto system must satisfy
\[
b = \frac{a+c}{2}+k\pi \mbox{ or } c = a + (2k+1)\pi \mbox{ for some } k \in \Z,
\]
and if $c = a + (2k+1)\pi$ then the solution cannot be linearly stable.
\end{lem}
\begin{proof}
At the vertex $v$, the condition of Equation~\ref{homogK} is simply
\[
\sin(c-b) +\sin(a-b) = 0,
\]
so either 
\[c-b = b-a +2k\pi
\]
 and the first possibility holds, or 
 \[(2k+1)\pi - (c-b) = b-a,
 \]and the second possibility holds. To see that $c=a+(2k+1)\pi$ does not result in a linearly stable system, we compute the Jacobian matrix of the system, ordering the vertices of $G$ starting with $b,a,c$. Let $\sigma_v$ denote the off diagonal row sum of the row corresponding to vertex $v$, and choose coordinates so $b=0$. Using that $\cos(a+(2k+1)\pi) = -\cos(a)$, we see that if vertices $a$ and $c$ are not connected, the top-left $3 \times 3$ submatrix of the Jacobian of the system is
\begin{equation}\label{PSD}
\left[ \!\begin{array}{ccc}
-\cos(c) -\cos(a)&\cos(a)&\cos(c)\\
\cos(a) &-\sigma_a & 0 \\
\cos(c) & 0 & -\sigma_c
\end{array}\! \right]
= 
\left[ \!\begin{array}{ccc}
0&\cos(a)&-\cos(a)\\
\cos(a) &-\sigma_a & 0 \\
-\cos(a) & 0 & -\sigma_c
\end{array}\! \right].
\end{equation}

Recall that there is a variant of Sylvester's theorem to check if a matrix is negative semidefinite (e.g.\S 6.3 of \cite{strang}): a symmetric matrix $M$ is negative semidefinite if and only if all odd principal minors are non-positive and all even principal minors are non-negative. The principal minors of the $3 \times 3$ submatrix above are 
\[
\{0, -\sigma_a, -\sigma_c, -\cos^2(a), \sigma_a\cdot\sigma_c, \cos^2(a)\cdot(\sigma_a+\sigma_c)\}
\]
By Lemma~\ref{glue2}, the Jacobian is a weighted negative Laplacian, and in particular $\sigma_a$ and $\sigma_c$ are non-negative.
 So for the $ 3 \times 3$ principal minor to be non-positive, we must have  
 \[
 \cos^2(a)\cdot(\sigma_a+\sigma_c) = 0 \Rightarrow \cos^2(a) = 0 \mbox{ or }\sigma_a=0=\sigma_c.
 \]
 
 If $\cos(a)=0$ then the matrix has a zero row yielding a zero eigenvalue. It is impossible for all remaining eigenvalues to be negative, because the submatrix resulting from deleting the first row and column still has all row sums equal to zero, so is itself singular. In particular, $\cos(a)=0$ results in a Jacobian matrix with at least two zero eigenvalues. 
 
 Next we consider the situation where $\sigma_a=0=\sigma_c$. The weighted Laplacian factors as $\widetilde{B_G}\cdot \widetilde{B_G}^T$ with $\widetilde{B_G}$ as in Lemma~\ref{glue2}, so a diagonal entry can be zero only if the corresponding row of $\widetilde{B_G}$ is the zero row. But since $\sigma_a = 0 = \sigma_c$, this would therefore imply that $\widetilde{B_G}$ has two zero rows. Since 
 \[
 \rank(A \cdot B) \le \min\{\rank(A),\rank(B)\},
 \]
 the kernel of the Jacobian has dimension at least two, so there are at least two zero eigenvalues. This settles the case when vertices $a$ and $c$ do not share an edge. 

To conclude, suppose $G$ has an edge $\overline{ac}$. In this case, the $(2,3)$ and $(3,2)$ entries of the matrix in Equation~\ref{PSD} are $-1$, and the $3 \times 3$ principal minor  is $\cos^2(a)\cdot(\sigma_a+\sigma_c+2)$, which is nonpositive only if $\cos(a)=0$. This case has been ruled out by the reasoning above.  
\end{proof}

\noindent Notice that in Figure 2, a pentagon and a triangle sharing a single common edge have an exotic solution, and in Figure 3, a pentagon and a $K_4$ sharing a common edge have an exotic solution. 

\begin{exm}\label{K5pent}
On eight vertices, a pentagon and a $K_5$ sharing a common edge as below have an exotic solution. This hints at a more general result, which will appear as Theorem~\ref{KnC5exotic}.
\begin{figure}[h]
\includegraphics[width=5.0in]{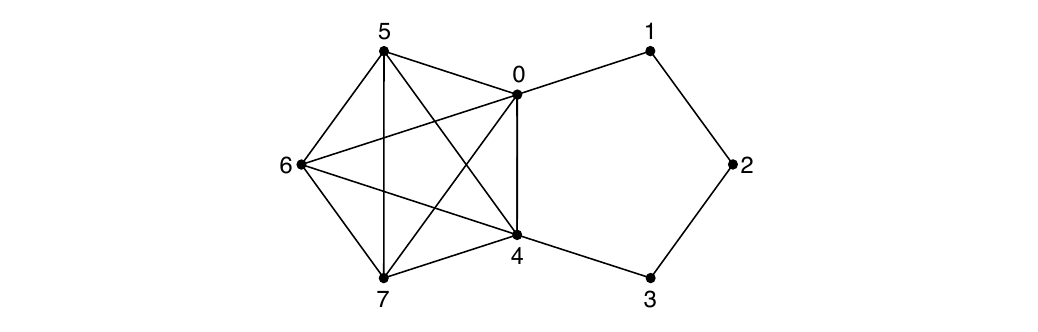}
\caption{An 8-vertex graph with exotic solutions}
\end{figure}

\vskip .1in
\noindent First, let $\theta_i = 0$ for $i \in \{5,6,7\}$. By Lemma~\ref{2valent}, the $\theta_i$ for $i \in \{1,2,3\}$ yield equations 
\[
\begin{array}{c}
 \theta_1 = (\theta_0+\theta_2)/2.  \\ 
\theta_2 = (\theta_1+\theta_3)/2.\\
 \theta_3 = (\theta_2+\theta_4)/2. \\
\end{array}
\]
\vskip .1in
\noindent To simplify notation, let $\theta_4 = \alpha$ and $\theta_3 = \beta$. We claim that setting $\theta_2=\pi$ and $\theta_0 = -\alpha$ yields an exotic state. To see this, first note that with these values
\[
\alpha = 2\beta-\pi.
\]
We need to show that the corresponding Jacobian matrix has all but one eigenvalue negative (recall that there is always at least one zero eigenvalue). Consider the remaining equations. At vertex $4$, we have 
\[
\begin{array}{ccc}
0& =& \sin(-\alpha-\alpha) +\sin(0-\alpha)+\sin(0-\alpha)+\sin(0-\alpha)+\sin(\beta-\alpha) \\
& = &-\sin(2 \alpha) -3\sin(\alpha) +\sin(\beta-\alpha)\\
&= & \sin(\beta)+3\sin(2\beta)-\sin(4\beta)\\
& = & \sin(\beta)(1 + 6\cos(\beta) -4\cos(\beta)(\cos^2(\beta)-\sin^2(\beta)))\\
& = & \sin(\beta)(1+10\cos(\beta)-8\cos^3(\beta))
\end{array}
\]
Since we need the last quantity to be zero, we either have $\beta \in \{ 0 , \pi \}$, which turns out to be impossible, or a solution to 
\[
1+10\cos(\beta)-8\cos^3(\beta)= 0.
\]
Letting $z=2\cos(\beta)$, we seek a root of 
\[
p(z) = 1+5z-z^3 \mbox{ with }z \in[-2..2]. 
\]
By Sturm's theorem $($see \S 2.2.2 of \cite{BPR}$)$, the number of roots in $(a,b]$ is 
\[
V(a)-V(b), \mbox{ where } V(t) = \sharp \mbox{ of sign changes in } [(p_0(t), p_1(t), p_2(t), p_3(t)],
\]
and $p_0=p(z), p_1=p'(z)$, and $p_i$ is the negative remainder on dividing $p_i$ by $p_{i-1}$. 

We find there is a unique value $(\sim -.2)$ of $z=2\cos(\beta)$ in $[-2,2]$. 
A computation shows that at the equilibrium point above, the Jacobian matrix is negative semidefinite, with a single zero eigenvalue. One of the exotic solutions arises as below, with $\alpha \sim 11.5728^{\circ} =  .064\pi$ radians.
\begin{equation}\label{eq5cycle}\begin{array}{ccc}
\theta_0 & = & -\alpha \\
\theta_1 & = & \frac{\pi - \alpha}{2} \\
\theta_2 & = & \pi\\
\theta_3 & = & \frac{\pi + \alpha}{2} \\
\theta_4 & = & \alpha \\
\theta_i & = & 0 \mbox{ for } i \in \{5,6,7\} 
\end{array}
\end{equation}
\end{exm}
\noindent The previous example generalizes, but we need a preparatory lemma.
\begin{lem}\label{glue1}
Let $G$ be a simple, connected graph on $d$-vertices and $C_5$ a chordless 5-cycle. Consider the graph $G'$ obtained by connecting the vertices of an edge $E$ of $C_5$ to every vertex of $G$. Then there is a unique value of $\alpha > 0$ such that assigning $\theta_i =0$ at all vertices of $G$ and $\{\theta_0,\ldots, \theta_4\}$ as in Equation~\ref{eq5cycle} yields an equilibrium solution to Equation~\ref{homogK} for $G'$.
\end{lem}
\begin{proof} The main point is that the pattern in Example~\ref{K5pent} generalizes. We choose a vertex labelling to simplify notation. Label the angles at the vertices of the $C_5$ graph as in Example~\ref{K5pent} 
\[
\{\theta_0, \ldots, \theta_4\}.
\]
Set the angles $\{\theta_5, \ldots\theta_{d+4}\}$ at the vertices of the graph $G$  to be zero, and define the remaining angles via
\begin{equation}\label{AnglesExotic}
\begin{array}{ccc}
\theta_i & = & \alpha + i \cdot \frac{(\pi - \alpha)}{2}, \mbox{ for }i \in \{0, \ldots 4\}.
\end{array}
\end{equation}
The parameter $\alpha$ is a function of $d$, obtained as in Example~\ref{K5pent}, but with a slight modification. Label the vertices of $E$ as $v_0$ and $v_4$. Then the calculation of $\theta_4$ changes very simply; we have
\[
\arraycolsep=1.3pt\def\arraystretch{1.5}
\begin{array}{rll}
0& =& \sin(-\alpha-\alpha) +d \cdot \sin(0-\alpha)+\sin(\frac{\pi+\alpha}{2}-\alpha) \\
& = &-\sin(2 \alpha) -d\sin(\alpha) +\sin(\frac{\pi-\alpha}{2})\\
& = &-2\sin(\alpha)\cos(\alpha) -d\sin(\alpha) +\cos(\frac{\alpha}{2})\\
& = &-2(2\sin(\frac{\alpha}{2})\cos(\frac{\alpha}{2}))\cos(\alpha)-d(2 \sin(\frac{\alpha}{2})\cos(\frac{\alpha}{2})) +\cos(\frac{\alpha}{2})\\
& = &\cos(\frac{\alpha}{2})\cdot (-4\sin(\frac{\alpha}{2})\cos(\alpha)-2d\sin(\frac{\alpha}{2}) +1)
\end{array}
\]
Since $\cos(\frac{\alpha}{2})=0$ does not yield a solution, using the identity  
\[
\cos(\alpha) = \cos^2\Big(\frac{\alpha}{2}\Big)-\sin^2\Big(\frac{\alpha}{2}\Big) = 1-2\sin^2\Big(\frac{\alpha}{2}\Big)
\]
and writing $w=\sin(\frac{\alpha}{2})$ yields
\[\arraycolsep=1.3pt\def\arraystretch{1.5}
\begin{array}{rll}
0& = &-4\sin(\frac{\alpha}{2})(1-2\sin^2(\frac{\alpha}{2}))-2d\sin(\frac{\alpha}{2}) +1\\
& = &-4w(1-2w^2)-2dw +1\\
& = &8w^3 -(4+2d)w +1\\
\end{array}
\]
Substituting $y=2w$ yields the depressed cubic $1-(d+2)y+y^3$, and 
\[
V(t) = \Big[t^3-(d+2)t+1, 3t^2-d-2, \frac{-2d-4}{3}t+1,d+2 -\frac{27}{4(d+2)^2}\Big],
\]
By Sturm's theorem, $1-(d+2)z+z^3$ has a single real root $\alpha \in[-2,2]$ when $d \ge 3$; when $d =1$ there are 3 roots in $[-2,2]$, and when $d=2$ there are 2 roots in $[-2,2]$. In both cases only one of the roots yields a linearly stable solution. 
\end{proof}

\begin{exm}
Gluing a triangle $(d=3)$ to $C_5$ yields Example~\ref{K5pent}. To two decimals the roots of $1-10w+8w^3$ are $\{1.06, .10, -1.16\}$ so \[\sin(\frac{\alpha}{2})\simeq.10 \Rightarrow \frac{\alpha}{2} \simeq .032\pi \mbox{ radians} \Rightarrow \alpha \simeq .064\pi \mbox{ radians},
\]agreeing with Example~\ref{K5pent}. The parameter $\alpha$ depends only on the number of vertices of $G$. 
\end{exm}
\noindent Lemmas \ref{glue2} and Lemma~\ref{glue1} lay the groundwork for producing families of graphs with exotic solutions which are not twisted stable states. The next result illustrates this technique.

\begin{thm}\label{KnC5exotic}
Let $G$ be a graph on $d$-vertices and $C_5$ a chordless 5-cycle. Consider the graph $G'$ obtained by choosing an edge $E$ of $C_5$, and connecting the vertices of $E$ to every vertex of $G$. Then $G'$ admits an exotic solution.
\end{thm}
\begin{proof}
This follows from Lemma~\ref{glue2}, Lemma~\ref{glue1}, and an application of Sturm's theorem. By Lemma~\ref{glue1} we have an equilibrium point. By Lemma~\ref{glue2}, if the weights 
\[
\cos(\theta_i -\theta_j)
\] 
are positive, then the equilibrium point is linearly stable. For the angles $\theta_i$ in Equation~\ref{eq5cycle}, this requires knowing the value of $\alpha$ appearing in Lemma~\ref{glue1}, which follows by applying Sturm's theorem to the interval $[0,2]$: $\alpha$ is a small positive number which is decreasing as $d$ increases. Applying this to the angles in Equation~\ref{AnglesExotic} shows they all satisfy  $\cos(\theta_i -\theta_j) >0$.
\end{proof}

\begin{rmk} The above construction can be carried out more generally, by gluing an arbitrary graph $G$ on $d$-vertices to an $n\ge 5$-cycle $C_n$. Label the angles at the vertices of the $C_n$ graph as 
\[
\{\theta_0, \ldots, \theta_{n-1}\},
\]
and the angles at the vertices of $G$ as 
\[
\{\theta_n, \ldots, \theta_{n+d-1}\}.
\]
Glue $G$ and $C_n$ along the edge connecting vertices $0$ and $n-1$, set the angles $\theta_n, \ldots, \theta_{n+d-1}$
to be zero, and define the remaining angles via
\[\begin{array}{ccc}
\theta_i & = & \alpha + i \cdot \frac{2(\pi - \alpha)}{n-1}, \mbox{ for }i \in \{0, \ldots n-1\}.
\end{array}
\]
The parameter $\alpha$ is a function of both $d$ and $n$, and in the local computation at vertex $0$ where gluing occurs, we obtain relations which involve $\sin(n \cdot \alpha)$ and $\sin(\alpha)$, leading to an expression in terms of Chebyshev polynomials. We leave this for the interested reader. 
\end{rmk}
\vskip .1in
 In the next section, we give more examples of exotic solutions. As S. Strogatz pointed out to us, an example of a stable exotic solution on a dodecahedron appears in \cite{Udeigwe}. One interesting question is the interplay between those graphs which have a positive dimensional solution set, and those graphs with exotic solutions, tabulated at the end of \S 2. Our computations indicate that for SCT graphs with 7 vertices, none of the graphs with exotic solutions are among the graphs with positive dimensional solutions. 

This is not the case for graphs with 8 vertices, where there is an overlap shown in Table 1.
Six of the graphs with exotic solutions also have a positive dimensional component. One of those six is Example ~\ref{K5pent}, and for five of the six, the positive dimensional component can be identified using Theorem~\ref{PDsolns}.
\pagebreak

\section{Computational Methods: The {\tt M2} package {\tt Oscillator.m2}}
\begin{exm}\label{comp1} In Example~\ref{Special8} we saw a non-cycle having twisted stable states:
\begin{verbatim}
i1 : needsPackage "Oscillators";

i2 : needsPackage "NautyGraphs";

i3 : G= graph{
         {0,4},{1,4},{1,5},
         {3,5},{2,6},{3,6},
         {0,7},{2,7},{4,5},
         {5,6},{6,7},{4,7}}
         
--input the graph G, as a list of edges, this is the format for
--the NautyGraphs package

o3 = Graph{0 => {4, 7}      }
           1 => {4, 5}
           2 => {6, 7}
           3 => {5, 6}
           4 => {0, 1, 5, 7}
           5 => {4, 1, 3, 6}
           6 => {5, 3, 2, 7}
           7 => {0, 4, 2, 6}
           
i4 : getExoticSolutions G

-- 115.674 seconds elapsed
-- fourd extra exotic solutions --

--coordinates of points (x_1..x_n-1, y_1..y_n-1); (x_0,y_0)=(1,0) is omitted.
+-+-+--+----+-----+-----+----+--+--+-+-----+-----+-----+-----+
|1|1|1 |1   |1    |1    |1   |0 |0 |0|0    |0    |0    |0    |
+-+-+--+----+-----+-----+----+--+--+-+-----+-----+-----+-----+
|0|0|-1|.707|-.707|-.707|.707|-1|1 |0|-.707|-.707|.707 |.707 |
+-+-+--+----+-----+-----+----+--+--+-+-----+-----+-----+-----+
|0|0|-1|.707|-.707|-.707|.707|1 |-1|0|.707 |.707 |-.707|-.707|
+-+-+--+----+-----+-----+----+--+--+-+-----+-----+-----+-----+

--angles (in degrees, first angle is always zero, and is omitted)
+---+---+---+---+---+---+---+
|0  |0  |0  |0  |0  |0  |0  |
+---+---+---+---+---+---+---+
|270|90 |180|315|225|135|45 |
+---+---+---+---+---+---+---+
|90 |270|180|45 |135|225|315|
+---+---+---+---+---+---+---+

\end{verbatim}
\end{exm}

\pagebreak
\begin{exm}\label{comp2} Next we compute the solutions for Example~\ref{K5pent}, and for a five cycle:
\begin{verbatim}
i14 :  K5C5= graph{{0,1},{0,2},{0,3},{0,4},{1,2},{1,3},{1,4},
                   {2,3},{2,4},{3,4},{0,5},{5,6},{6,7},{1,7}}

o14 = Graph{0 => {1, 2, 3, 4, 5}}
            1 => {0, 2, 3, 4, 7}
            2 => {0, 1, 3, 4}
            3 => {0, 1, 2, 4}
            4 => {0, 1, 2, 3}
            5 => {0, 6}
            6 => {5, 7}
            7 => {1, 6}
i15 : getExoticSolutions K5C5
-- 99.7506 seconds elapsed
-- found extra exotic solutions --
--coordinates of points (x_1..x_n-1, y_1..y_n-1); (x_0,y_0)=(1,0) is omitted.

+----+---+---+---+----+----+-----+-----+-----+-----+-----+-----+-----+-----+
|1   |1  |1  |1  |1   |1   |1    |0    |0    |0    |0    |0    |0    |0    |
+----+---+---+---+----+----+-----+-----+-----+-----+-----+-----+-----+-----+
|.919|.98|.98|.98|.101|-.98|-.298|.393 |.201 |.201 |.201 |-.995|-.201|.954 |
+----+---+---+---+----+----+-----+-----+-----+-----+-----+-----+-----+-----+
|.919|.98|.98|.98|.101|-.98|-.298|-.393|-.201|-.201|-.201|.995 |.201 |-.954|
+----+---+---+---+----+----+-----+-----+-----+-----+-----+-----+-----+-----+

--angles (in degrees, first angle is always zero, and is omitted)
+-------+-------+-------+-------+-------+-------+-------+
|0      |0      |0      |0      |0      |0      |0      |
+-------+-------+-------+-------+-------+-------+-------+
|23.1455|11.5728|11.5728|11.5728|275.786|191.573|107.359|
+-------+-------+-------+-------+-------+-------+-------+
|336.854|348.427|348.427|348.427|84.2136|168.427|252.641|
+-------+-------+-------+-------+-------+-------+-------+

For the five cycle we suppress the x_i,y_i coordinates, printing only angles
i36 : getExoticSolutions Pent
---- doing graph Graph{0 => {1, 4}} with count 4 --------------
                       1 => {0, 2}
                       2 => {1, 3}
                       3 => {2, 4}
                       4 => {0, 3}
+---+---+---+---+
|72 |144|216|288|
+---+---+---+---+
|288|216|144|72 |
+---+---+---+---+
|0  |0  |0  |0  |
+---+---+---+---+
\end{verbatim}
        \end{exm}
        \pagebreak
        \begin{exm}\label{comp3} For graphs with 7 or fewer vertices, there are at most 2 nonstandard exotic solutions. This is no longer the case for graphs with 8 or more vertices. Below we compute solutions for two pentagons sharing an edge (an example with eight vertices) and a pentagon and hexagon sharing an edge (which has nine vertices). 
        
        The computation shows that there are, respectively, four and six nonstandard solutions. Recall the first angle is 0, and is not printed. 
        \begin{verbatim}
i10: PentPent = graph{{0,1},{1,2},{2,3},{3,4},{4,0},{0,5},{5,6},{6,7},{7,1}};

i11 : getExoticSolutions PentPent

--Display of coordinates of the points suppressed, displaying only angles
--angles (in degrees, first angle is always zero, and is omitted)

+-------+-------+-------+-------+-------+-------+-------+
|310.376|322.782|335.188|347.594|77.594 |155.188|232.782|
+-------+-------+-------+-------+-------+-------+-------+
|49.6241|127.218|204.812|282.406|12.406 |24.8121|37.2181|
+-------+-------+-------+-------+-------+-------+-------+
|310.376|232.782|155.188|77.594 |347.594|335.188|322.782|
+-------+-------+-------+-------+-------+-------+-------+
|0      |0      |0      |0      |0      |0      |0      |
+-------+-------+-------+-------+-------+-------+-------+
|49.6241|37.2181|24.8121|12.406 |282.406|204.812|127.218|
+-------+-------+-------+-------+-------+-------+-------+

i12 : HexPent = graph{{0,1},{1,2},{2,3},{3,4},{4,0},{0,5},{5,6},{6,7},{7,8},{8,1}};

i13 : getExoticSolutions HexPent

--Display of coordinates of the points suppressed, displaying only angles
--angles (in degrees, first angle is always zero, and is omitted)

+-------+-------+-------+-------+-------+-------+-------+-------+
|315.613|326.709|337.806|348.903|63.1225|126.245|189.368|252.49 |
+-------+-------+-------+-------+-------+-------+-------+-------+
|2.64051|91.9804|181.32 |270.66 |72.5281|145.056|217.584|290.112|
+-------+-------+-------+-------+-------+-------+-------+-------+
|44.3875|33.2906|22.1937|11.0969|296.877|233.755|170.632|107.51 |
+-------+-------+-------+-------+-------+-------+-------+-------+
|307.614|230.71 |153.807|76.9034|349.523|339.045|328.568|318.091|
+-------+-------+-------+-------+-------+-------+-------+-------+
|52.3863|129.29 |206.193|283.097|10.4773|20.9545|31.4318|41.9091|
+-------+-------+-------+-------+-------+-------+-------+-------+
|357.359|268.02 |178.68 |89.3399|287.472|214.944|142.416|69.8876|
+-------+-------+-------+-------+-------+-------+-------+-------+
|0      |0      |0      |0      |0      |0      |0      |0      |
+-------+-------+-------+-------+-------+-------+-------+-------+
         \end{verbatim}

        \end{exm}
         \begin{exm}\label{comp4} 
        There are only two examples of SCT graphs on 8 vertices having an exotic solution where the Jacobian matrix has some of the off-diagonal entries $\cos(\theta_j - \theta_i)$ negative. We illustrate with one of these below.
\begin{figure}[h]
\vskip -.15in
\includegraphics[width=5in]{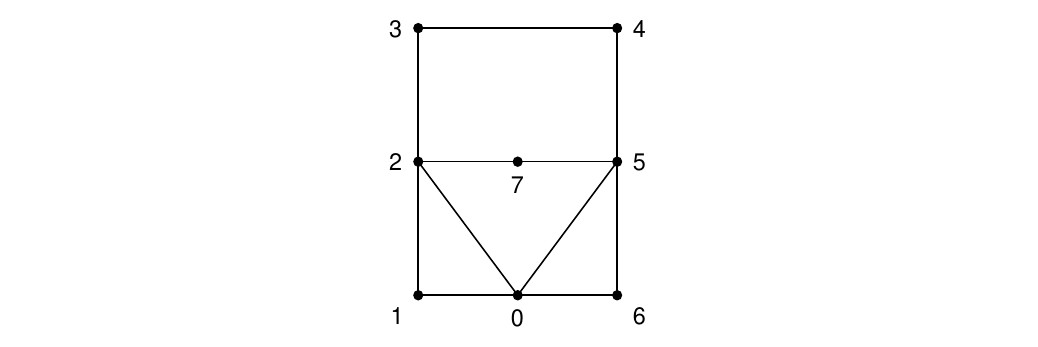} 
\vskip -.2in
\caption{One of two SCT graphs $G$ with negative off-diagonal entries in Jacobian.}
\vskip -.05in
\end{figure}
        \begin{verbatim}
i47 : G = graph {{0,1},{1,2},{2,3},{3,4},{4,5},{5,6},{6,0},{0,5},{0,2},{5,7},{2,7}};

i48 : getExoticSolutions G;
-- found extra exotic solutions --
--coordinates of points (x_1..x_n-1, y_1..y_n-1); (x_0,y_0)=(1,0) is omitted.
+----+-----+-----+-----+-----+----+--+-----+-----+-----+-----+-----+-----+-+
|.623|-.222|-.900|-.900|-.222|.623|-1|-.781|-.974|-.433|.433 |.974 |.781 |0|
+----+-----+-----+-----+-----+----+--+-----+-----+-----+-----+-----+-----+-+
|.623|-.222|-.900|-.900|-.222|.623|-1|.781 |.974 |.433 |-.433|-.974|-.781|0|
+----+-----+-----+-----+-----+----+--+-----+-----+-----+-----+-----+-----+-+
|1   |1    |1    |1    |1    |1   |1 |0    |0    |0    |0    |0    |0    |0|
+----+-----+-----+-----+-----+----+--+-----+-----+-----+-----+-----+-----+-+

--angles (in degrees, first angle is always zero, and is omitted)
+-------+-------+-------+-------+-------+-------+---+
|308.571|257.143|205.714|154.286|102.857|51.4286|180|
+-------+-------+-------+-------+-------+-------+---+
|51.4286|102.857|154.286|205.714|257.143|308.571|180|
+-------+-------+-------+-------+-------+-------+---+
|0      |0      |0      |0      |0      |0      |0  |
+-------+-------+-------+-------+-------+-------+---+
 -- 54.6913 seconds elapsed
i49 :     sub(JC, matrix {pts_0})

o49 = | -.801938 .62349   -.222521 0        0        -.222521 .62349   0        |
      | .62349   -1.24698 .62349   0        0        0        0        0        |
      | -.222521 .62349   -1.24698 .62349   0        0        0        .222521  |
      | 0        0        .62349   -1.24698 .62349   0        0        0        |
      | 0        0        0        .62349   -1.24698 .62349   0        0        |
      | -.222521 0        0        0        .62349   -1.24698 .62349   .222521  |
      | .62349   0        0        0        0        .62349   -1.24698 0        |
      | 0        0        .222521  0        0        .222521  0        -.445042 |
         \end{verbatim}
         \vskip -.1in
The Jacobian matrix above  corresponds to the first solution. Letting $\alpha = 51.4286^{\circ} \simeq \frac{2\pi}{7}$, that solution has angles $\theta_0=0$, $\theta_i = -i \cdot \alpha$ for $i \in \{1,..,6\}$, and $\theta_7 = \pi$.
        \end{exm}
          
\section{Summary and future directions}\label{conclusion}
In this work, we study systems of homogeneous Kuramoto oscillators from an algebraic viewpoint. By translating into a system of algebraic equations, we obtain insight into the structure of possible linearly stable solutions. Our focus is on simple, connected graphs with all vertices of degree at least two, which we call SCT graphs. \vskip -.05in
\vskip -.05in
\subsection{Main results} The main results of this paper are the following:
\vskip .05in
\noindent$\bullet$ In Theorem~\ref{PDsolns}, we give sufficient conditions for $I_G$ to have a positive dimensional component in the solution set. This is important because no solution lying on a positive dimensional component can be a linearly stable solution. 
 \vskip .05in
\noindent$\bullet$  In Theorem~\ref{SegreAssPrime}, we identify the ideal $I_\Sigma$ of the Segre variety of $\pp^1 \times \pp^{n-1}$ as an associated prime of $I_G$. We show that all standard solutions lie on the Segre variety. 
\vskip .05in
\noindent$\bullet$ For $G$ having $V \le 8$ vertices, we determine all SCT graphs which admit exotic solutions, and which admit positive dimensional solutions. There are, respectively, $\{3,11,61,507,7442\}$ isomorphism classes of SCT graphs with $V \in \{4,5,6,7,8\}$. Of the graphs on 8 vertices, 81 have exotic solutions, and all of them--with one exception--has an induced cycle of length at least five. 
\vskip .05in
\noindent$\bullet$ The cyclic graph $C_n$ on $n$-vertices always admits exotic solutions, the {\em twisted stable states}, where each angle is a periodic translate of an adjacent angle. Theorem~\ref{KnC5exotic} gives a general method to construct graphs with exotic solutions which are not twisted stable states. 
\vskip -.2in
\subsection{Future directions} This paper raises a number of questions for additional investigation.
\vskip .05in
\noindent$\bullet$ {\bf Structure of the ideal $I_G$.} Building on the work of \S 2, it would be interesting to further analyze the irreducible decomposition of $I_G$. Preliminary results indicate that it may be a radical ideal, and we are conducting further computations to explore the situation. \vskip .05in
 \vskip .05in
\noindent$\bullet$ {\bf Gluing graphs.} In algebraic topology, a standard construction is the {\em Mayer-Vietoris} sequence (see, e.g. \S 4.4.1 of \cite{TDAbook}). Given two topological spaces $X_1$ and $X_2$, select a common subspace $Y$. Since $Y \subseteq X_i$, we can identify $X_1$ and $X_2$ along $Y$. The Mayer-Vietoris sequence relates the topology of $X_1$, $X_2$ and $Y$ to the topology of $X_1 \cup_Y X_2$ (we have ``glued'' $X_1$ and $X_2$ together along a common subgraph). The results in \S 3 suggest studying systems of oscillators using the Mayer-Vietoris sequence, and we are currently at work on this project. \vskip .05in
\vskip .05in
\noindent$\bullet$ {\bf Structure of exotic solutions.} Of the 81 graphs on eight vertices with exotic solutions, all but one have a pair of exotic solutions. As illustrated in Example~\ref{comp3}, a pair of $C_5$'s glued on an edge admits four exotic solutions; a $C_6$ glued on an edge to a $C_5$ admits six exotic solutions. Can we characterize the number of exotic solutions from graph-theoretic data?
\vskip .05in
\noindent$\bullet$ {\bf Real versus complex solutions.} It is possible that a component $X \subseteq {\mathbb V}(I_G)$ has positive complex dimension, but for $X$ to contain only finitely many real points. Evaluating the Jacobian matrix at a real point $p \in X$ will yield a matrix of rank $< n-1$, because the tangent space $T_p(X)$ is positive dimensional. In particular, such a point can never be linearly stable. 
Nevertheless, since linear stability is a sufficient but not necessary condition for stability, analyzing stability for such points is an interesting problem for future work. 
\vskip .05in
\noindent{\bf Acknowledgments.} The January 2025 release of {\tt https://macaulay2.com} incorporates the {\tt Oscillators.m2} package written to support computations here, and is also available at {\tt https://github.com/mikestillman/oscillators}.
Our collaboration began while the second two authors were visitors at Oxford, supported, respectively, by Leverhulme and Simons Fellowships. We thank those foundations for their support, and the Oxford Mathematical Institute for providing a wonderful working atmosphere. Thanks to Steve Strogatz, Alex Townsend, and two anonymous referees for helpful comments.
\pagebreak
%

\bibliographystyle{amsalpha}

\end{document}